\documentclass[12pt,a4paper]{amsart}

\usepackage[latin1]{inputenc}
\usepackage{graphicx}
\usepackage{amssymb, amsmath}
\usepackage{color}
\usepackage{geometry}
\usepackage[shortlabels]{enumitem}
\usepackage{amsthm}

\usepackage[dvipsnames]{xcolor}
\usepackage[colorlinks=true,linkcolor=blue,urlcolor=blue,citecolor=blue]{
hyperref}
\usepackage{mathpazo}
\usepackage{domitian}
\usepackage[T1]{fontenc}

\newtheorem{theorem}{Theorem}[section]

\newtheorem{lemma}[theorem]{Lemma}
\newtheorem{corollary}[theorem]{Corollary}
\newtheorem{proposition}[theorem]{Proposition}

\theoremstyle{definition}
\newtheorem{example}[theorem]{Example}
\newtheorem{remark}[theorem]{Remark}

\newcommand{\T}{\mathbb{T}}
\newcommand{\R}{\mathbb{R}}
\newcommand{\Q}{\mathbb{Q}}
\newcommand{\E}{\mathbb{E}}

\newcommand{\N}{\mathbb{N}}

\newcommand{\zz}{\mathbb{Z}}
\newcommand{\hh}{\mathcal{H}}
\newcommand{\C}{\mathbb{C}}

\DeclareMathOperator{\re}{Re}
\DeclareMathOperator{\supp}{supp}

\DeclareFontFamily{U}{wncy}{}
\DeclareFontShape{U}{wncy}{m}{n}{<->wncyr10}{}
\DeclareSymbolFont{mcy}{U}{wncy}{m}{n}
\DeclareMathSymbol{\Sh}{\mathord}{mcy}{"58}
\usepackage{amsmath,amsfonts,amssymb,amsthm}
\usepackage{mathtools}

\author[Carando]{Daniel Carando}
\address{Departamento de Matem{\'a}tica\\ Facultad de Cs. Exactas y Naturales\\
		Universidad de Buenos Aires and IMAS-UBA-CONICET\\ Int.~G{\"u}iraldes s/n, 1428\\ Buenos Aires, Argentina.}
\email{dcarando@dm.uba.ar}

\author[Defant]{Andreas Defant}
\address{Institut f\"ur Mathematik\\
Carl von Ossietzky Universit\"at\\
26111 Oldenburg, Germany.}
\email{defant@mathematik.uni-oldenburg.de}

\author[Marceca]{Felipe Marceca}
\address{Department of Mathematics\\ King's College London\\ United Kingdom}
\email{felipe.marceca@kcl.ac.uk}

\author[Schoolmann]{Ingo Schoolmann}
\address{Institut f\"ur Mathematik\\
Carl von Ossietzky Universit\"at\\
26111 Oldenburg, Germany.}
\email{ingo.schoolmann@uni-oldenburg.de}

\author[Sevilla-Peris]{Pablo Sevilla-Peris}
\address{Instituto Universitario de Matem\'atica Pura y Aplicada\\
Universitat Polit\`{e}cnica de Val\`encia\\ cmno Vera s/n, 46022\\
Val\`encia, Spain}
\email{psevilla@mat.upv.es}

\title[Hypercontractivity and strips for general Dirichlet series]{Hypercontractivity and  strips of convergence in Hardy spaces of general Dirichlet series}

\thanks{The first author was partially supported by CONICET PIP  11220200102366CO and
ANPCyT PICT 2018-04104.
The third author gratefully acknowledges support from EPSRC: NIA EP/V002449/1 and the Austrian Science Fund (FWF): Y 1199.
supported by grant PID2021-122126NB-C33 funded by MCIN/AEI/10.13039/501100011033 and by `ERDF A way of making Europe', and by GV Project AICO/2021/170}



\begin{document}

\maketitle

\begin{abstract}
    For a general Dirichlet series $\sum a_n e^{-\lambda_n s}$ with frequency $\lambda=(\lambda_n)_n$, we study  how horizontal translation (i.e. convolution with a Poisson kernel) improves its integrability properties. We characterize hypercontractive frequencies in terms of their additive structure answering some questions posed by Bayart. We also provide sharp bounds for the strips $S_p(\lambda)$ that encode the minimum translation necessary for series in the Hardy space $\mathcal H_p(\lambda)$ to have absolutely convergent coefficients.
\end{abstract}

%\tableofcontents

\section{Introduction and main results}

Given a frequency $\lambda = (\lambda_n)$, i.e. a strictly increasing sequence of non-negative real numbers tending to $\infty$,
a \emph{$\lambda$-Dirichlet series}  is a (formal) series of the form $D=\sum a_{n}e^{-\lambda_{n}s}$, where $s$ is a complex variable and the sequence  $(a_{n})$
(of so-called Dirichlet coefficients) belongs to $\C$.

A $\lambda$-Dirichlet series with finitely many non-zero coefficients will be called a \emph{$\lambda$-Dirichlet polynomial}. As in the ordinary case, for each $1\le p < \infty$ we can define the $p$-norm of a $\lambda$-Dirichlet polynomial $D$ with coefficients $a_1,\ldots,a_N$  as
\begin{equation}\label{pnorm}
  \|D\|_p= \lim_{T \to \infty} \Big(\frac{1}{2T}\int_{-T}^T \Big|\sum_{n=1}^N a_n e^{-\lambda_n it} \Big|^p dt\Big)^{1/p} \,.
\end{equation}
With this, the \emph{Hardy space $\mathcal{H}_p(\lambda)$} of $\lambda$-Dirichlet series is defined as the completion of the space of $\lambda$-Dirichlet polynomials with the above norm. There is another equivalent definition of these spaces involving Fourier analysis on groups (see \cite{Ba2} for the ordinary case $\lambda=(\log n)$ and  \cite[Definition~3.23 and Theorem~3.26]{DeSch} for the general case) that in some aspects is more suitable to deal with them.
We frequently use results that stem from this point of view, and refer  for a short  exposition of this construction to the preliminaries (in particular, for an appropriate definition of $\mathcal{H}_\infty(\lambda)$).

It is well known (and easy to check) that these spaces are decreasing, in the sense that if $1 \leq p < q < \infty$, then $\mathcal{H}_{q}(\lambda) \subseteq \mathcal{H}_{p}(\lambda)$. On the other hand, horizontal translations can take a Dirichlet series to a smaller Hardy space.
Given a $\lambda$-Dirichlet series $D= \sum a_{n} e^{-\lambda_{n} s}$ and $\sigma \in \mathbb{R}$, we define a new $\lambda$-Dirichlet series as
\begin{equation} \label{gallardo}
D_\sigma:=\sum a_n e^{-\lambda_n \sigma} e^{-\lambda_n s} \,,
\end{equation}
which we often will call the \textit{translation} of $D$ by $\sigma$.
By \cite[Theorem~4.7]{DeSch}, for each $\sigma >0$, the operator $\tau_{\sigma}$  given by $\tau_\sigma (D) = D_\sigma$ is well defined and a contraction when defined from $\mathcal{H}_{p} (\lambda)$ to $\mathcal{H}_{p} (\lambda)$ (and hence also to any $\mathcal{H}_{q}(\lambda)$ with $q \geq p$).
In this note we study the values $1 \leq p < q < \infty$ and $\sigma >0$ for which the operator $\tau_{\sigma} : \mathcal{H}_{p} (\lambda) \to \mathcal{H}_{q}(\lambda)$ is well defined and continuous. We have three main goals.

Our first goal is to describe \textit{hypercontractive} frequencies, that is, frequencies $\lambda$ such that the operator 
$\tau_{\sigma} : \mathcal{H}_{p} (\lambda) \to \mathcal{H}_{q}(\lambda)$ is well defined and continuous for every pair $1 \leq p < q < \infty$ and all $\sigma >0$. The term stems from regarding $\tau_{\sigma}$ as the Poisson semigroup in certain subspaces of $L^p(G)$ where $G$ is an ordered group (see \cite{DeSch,Ru} and also Section~\ref{secprel}).

In \cite[Proposition~5.4]{Ba}, Bayart showed that for a frequency  $\lambda$ such that
\begin{equation*}
      \lim_{\mu \to \infty} \frac{\log(\#\{ (n_1,\ldots,n_k) \colon   \lambda_{n_1}+\ldots+\lambda_{n_k} = \mu  \,       \} )}{\mu}=0,
\end{equation*}
the operator $\tau_{\sigma}$ maps $\hh_2(\lambda)$ into $\hh_{2k}(\lambda)$ for every $k\in \N$ and every $\sigma>0$. This points to a relation between hypercontractivity and a lack of additive structure of the frequency.

Delving deeper into this phenomenon we can provide a characterization of hypercontractivity under the mild growth condition $\limsup \tfrac{\log\log n}{\lambda_n}=0$. Recall that the \emph{$k$-th additive energy} of a set $A \subseteq \mathbb{R}$ is defined as 
\[
E_k(A)=\#\{(x,y)\in A^{2k} \colon \ x_1+\ldots+x_k=y_1+\ldots+y_k\} \,.
\]
See \eqref{corona} for a simple reformulation in terms of norms of Dirichlet polynomials.

\begin{theorem}\label{corchar}
	Let $\lambda$ be a frequency. The following are equivalent:
	\begin{enumerate}[label=(\roman*)]
		\item\label{iti} $\lambda$ is hypercontractive;
		\item\label{itii}  $\tau_\sigma$ maps $\hh_2(\lambda)$ into $\hh_{2k}(\lambda)$ for every $\sigma>0$ and every $k\in \N$.
	\end{enumerate}		
If additionally $\lambda$ satisfies that $\limsup \tfrac{\log\log n}{\lambda_n}=0$, both are also equivalent to:
\begin{enumerate}[label=(iii)]
		\item\label{itiii} For every $k\ge 2$ we have
		\[
		\lim_{j\to\infty} \sup_{A\subseteq \lambda\cap [j,j+1)}\frac{1}{j}\log \Big(\frac{E_k(A)^{1/k}}{\#A}\Big)=0\,.
		\]
	\end{enumerate}
\end{theorem}
Notice that  \ref{itii} implies that a finite union of hypercontractive frequencies is hypercontractive. It is also worth pointing out that the equivalence between \ref{itii} and \ref{itiii} holds also for each fixed $k$, as shown in Lemma~\ref{lemen}. Finally, we remark  that although \ref{itiii} may seem quite involved, it only appeals to the additive structure of~$\lambda$.

From \cite[Proposition~5.4]{Ba} and the implication \ref{itii} $\Rightarrow$ \ref{iti} we immediately get the following corollary.

\begin{corollary}\label{burrito}
  Let $\lambda$ be a frequency such that
  \begin{equation*}
      \lim_{\mu \to \infty} \frac{\log(\#\{ (n_1,\ldots,n_k) \colon   \lambda_{n_1}+\ldots+\lambda_{n_k} = \mu  \,       \} )}{\mu}=0,
  \end{equation*}
  for all $k\in \N$. Then, $\lambda$ is hypercontractive.
\end{corollary}
Dealing with this sufficient condition will be more practical to prove hypercontractivity for specific frequencies (see Section~\ref{secexa}).

Our second goal is, for a non-hypercontractive frequency $\lambda$, to study the region where the optimal $\sigma$ guaranteeing that $\tau_\sigma$ maps $\hh_p(\lambda)$ into $\hh_q(\lambda)$ lives.
We denote the abscissas of convergence and absolute convergence of a $\lambda$-Dirichlet series $D$ as
\begin{align*}
\sigma_{c}(D)&:=\inf\left \{ \sigma \in \R : \ D \text{ converges on } [\re s>\sigma] \right\},
\\ \sigma_{a}(D)&:=\inf\left \{ \sigma \in \R : \ D \text{ converges absolutely on } [\re s>\sigma] \right\}.
\end{align*}
The maximal width of the strip of convergence and non absolute convergence is denoted by $L(\lambda)$ (see \cite[Section~3]{Bohr1914}):
\[
L(\lambda):=\sup(\sigma_{a}(D) - \sigma_{c}(D))=\sigma_c\Big(\sum_{n\in\N} e^{-\lambda_n s}\Big)=\limsup_{n\to\infty} \frac{\log n}{\lambda_n}\,.
\]
In \cite[Corollary~5.5]{Ba}, it is proven that if $\sigma > \tfrac{k-1}{2}L(\lambda)$, then $\tau_\sigma$ maps $\hh_{2}(\lambda)$ into $\hh_{2k}(\lambda)$. The following result shows that the key magnitude is in fact $\tfrac{k-1}{2k}L(\lambda)$ and extends this to other values of $p$ and $q$, also answering \cite[Question~5.7]{Ba}.

\begin{proposition}\label{cormax}
	For every $1\le p \le 2$, $p\le q\le\infty$ and every frequency $\lambda$, if  $\sigma> L(\lambda)(\tfrac{1}{p}-\tfrac{1}{q})$, then  $\tau_\sigma$ maps $\hh_p(\lambda)$ into $\hh_q(\lambda)$.
\end{proposition}

In particular, as already noticed in \cite[Remark~5.15]{DeFVScSP}, $L(\lambda)=0$ implies hypercontractivity.
The case $p>2$ seems more involved as is common in the context of Nikolskii-type inequalities (see \cite[Section~6]{DiTi}).

\begin{remark}
If $p=2$ we can give a a simple proof of Proposition~\ref{cormax}. By the Hausdorff-Young inequality (see \eqref{hy2} in Section~\ref{sec-proofs}),
for a Dirichlet series $D=\sum a_n e^{-\lambda_n s}$ and $2\le q \leq \infty$ we have,
\begin{align*}
		\|\tau_\sigma D\|_q 
&
\le \Big(\sum_{n\ge 1} |a_n|^{q'}e^{-\sigma q' \lambda_n}\Big)^{1/q'}
\\&
		\le \Big(\sum_{n\ge 1} e^{-\sigma 2q' \lambda_n/(2-q')}\Big)^{1/q'-1/2}\Big(\sum_{n\ge 1} |a_n|^{2}\Big)^{1/2}
		\\&
\le (\sum_{n\ge 1} e^{-\sigma 2q' \lambda_n/(2-q')}\Big)^{1/q'-1/2} \|D\|_2,
\end{align*}
where we used H\"older's inequality with parameters $2/(2-q')$ and $2/q'$. So, $\tau_\sigma$ is bounded whenever $\sigma 2q'/(2-q') > L(\lambda)$.
\end{remark}

The bound $L(\lambda)(\tfrac{1}{p}-\tfrac{1}{q})$ in Proposition~\ref{cormax} cannot be improved. Moreover, the following proposition~shows that the boundedness of $\tau_\sigma$ can be attained at any point between 0 and $L(\lambda)(\tfrac{1}{p}-\tfrac{1}{q})$.

\begin{proposition}\label{propshift}
For each $\alpha\in[0,1]$, there exists a frequency $\lambda$ satisfying the following two properties: 
\begin{enumerate}[(a)]
    \item\label{propshifta} If $1 \leq p \le q \leq \infty$ and $\tau_\sigma$ maps $\hh_p(\lambda)$ into $\hh_q(\lambda)$, then $\sigma\ge \alpha (\tfrac{1}{p}-\tfrac{1}{q})L(\lambda)$;

    \item\label{propshiftb} If $1  \le p \leq 2$ and $p \leq q < \infty $, then for $\sigma \ge \alpha (\tfrac{1}{p}-\tfrac{1}{q})L(\lambda)$ the operator
$\tau_\sigma$ maps $\hh_p(\lambda)$ into $\hh_q(\lambda)$. 
\end{enumerate}
\end{proposition}

In \cite[Theorem~5.11]{Ba}, Bayart constructed a frequency $\lambda $ such that $\tau_\sigma$ maps $\hh_{2}(\lambda)$ into $\hh_{2k}(\lambda)$ if and only if $\sigma \ge \frac{k-1}{2k}$. He also asked in \cite[Question~5.12]{Ba} whether, for that frequency, $\tau_\sigma$ maps $\hh_{2}(\lambda)$ into $\hh_{p}(\lambda)$ if and only if $\sigma \ge \frac{p-2}{2p}$. Although we work with a slightly different frequency for notational convenience, the proof of Lemma~\ref{lemext} below, which is a main ingredient for the proofs of Propositions~\ref{propshift} and~\ref{example}, provides the tools to answer \cite[Question~5.12]{Ba} affirmatively.

Our third (and final) goal is to study strips of convergence. Define the $\mathcal{H}_p$-abscissa of $D$ as
\[
\sigma_p(D):=\inf\left \{ \sigma \in \R : \ D_\sigma \in \mathcal{H}_p(\lambda)  \right\} \,.
\]
where $D_\sigma$ stands for the translation defined in  \eqref{gallardo}.
Also define,
\begin{equation*}
    S_p(\lambda):=\sup\{\sigma_{a}(D) - \sigma_{p}(D): \ D \text{ Dirichlet series} \}
 =\sup_{D\in \mathcal{H}_p(\lambda)} \sigma_{a}(D),
\end{equation*}
where we understand $\infty-\infty$ to be $0$. By a closed-graph argument it is easy to check that $S_p(\lambda)$ is also the infimum over all $\sigma$ such that there exists a constant $C_\sigma$ satisfying
\begin{equation}\label{closedgraph}
  \sum |a_n| e^{-\lambda_n \sigma} \le C_\sigma \|D\|_p\,,
\end{equation}
for every $D=\sum a_n e^{-\lambda_n s}\in \mathcal{H}_p(\lambda)$.
It will be useful to encode the behaviour of $\tau_\sigma$ from a strip point of view by denoting
\begin{multline*}
 T_{q,p}(\lambda):=\sup\{\sigma_{q}(D) - \sigma_{p}(D): \ D \text{ Dirichlet series} \}
 \\=\inf\{\sigma>0 \colon \ \tau_\sigma: \hh_p(\lambda) \to \hh_q(\lambda) \text{ is bounded}\}.
\end{multline*}
We remark that again by a closed-graph argument, if $\tau_\sigma$ maps $\hh_p(\lambda) $ into $ \hh_q(\lambda)$, then it defines a bounded operator.
It follows from the definitions that $S_p(\lambda)$ and $T_{q,p}(\lambda)$ are decreasing in $p$, while $T_{q,p}(\lambda)$ is increasing in $q$ and $T_{\infty,p}(\lambda)\le S_{p}(\lambda)$. Note also that $\lambda$ being hypercontractive is just $T_{q,p}(\lambda)=0$ for every $1\leq p,q < \infty$.

As observed in \cite[Section~3.7]{DeSch}, we have
\begin{equation} \label{izagirre}
S_2(\lambda)=\frac{L(\lambda)}{2} \,.
\end{equation}
For ordinary Dirichlet series it is known that this holds for every $p$, that is, $S_p((\log n)_n)=1/2$ for every $1\le p \le \infty$ (see, for example, \cite{Ba2} and \cite[Theorem~12.11]{ElLibro}). This was extended to frequencies with similar algebraic structure in \cite[Corollary~3.35]{DeSch}. Hypercontractive frequencies also satisfy this, as the following remark shows.

\begin{remark} \label{bilbao}
From the definitions, for $1 \leq p < q$ we get
\begin{equation*}
    S_{q} (\lambda) \le S_{p} (\lambda) \le S_{q}(\lambda) + T_{q,p}(\lambda)  \,.
\end{equation*}
Hence, if $\lambda$ is hypercontractive, then  we have $S_{p}(\lambda) = S_{q}(\lambda)$ and as a consequence,
\[
S_{p} (\lambda) = \frac{L(\lambda)}{2} \,,
\]
for every $1 \leq p < \infty$.
\end{remark}

However, in \cite[Theorem~5.2]{Ba} Bayart constructs a frequency $\mu$ such that $S_1(\mu)=L(\mu)$,\footnote{A slightly different strip is used by Bayart, but the result holds in our context as well.} which in particular prevents hypercontractivity by the previous remark.

The following theorem~gives the general behaviour of $S_p(\lambda)$.

\begin{theorem}\label{rem}
For any frequency $\lambda$ we have
\begin{align}
  \frac{L(\lambda)}{2}\le S_p(\lambda)&\le \frac{L(\lambda)}{p}, \qquad \text{for } 1\le p\le 2 \,; \label{rem1}
\\  S_p(\lambda)&= \frac{L(\lambda)}{2}, \qquad \text{for } 2\le p <\infty \,.\label{rem2}
\end{align}
Moreover, for $1\le p\le 2$, we have that 
$p S_p(\lambda)$ is increasing in $p$, $S_p(\lambda)$ is a Lipschitz function of $p$ and for almost every $p$ 
\begin{equation}   \label{eqlip}
    -\frac{S_p(\lambda)}{p}\le \frac{d S_p(\lambda)}{d p} \le 0.
\end{equation}
\end{theorem}

\begin{remark} \label{bilbao-2}
Recall that, by Remark~\ref{bilbao}, if $T_{2,1}(\lambda)=0$  then $S_{1} (\lambda)= S_{2} (\lambda)$. This, in view of \eqref{rem2}, gives
\[
S_{p} (\lambda) = \frac{L(\lambda)}{2} \,,
\]
for every $1 \leq p < \infty$. As we will see later (see Corollary~\ref{corlil}), $T_{2,1}(\lambda)=0$ whenever $T_{q,p}(\lambda)=0$ for some $1\le p,2\le q$. So, in this case we also have $S_{p}(\lambda) = \frac{L(\lambda)}{2}$ for every $p$. The condition $T_{4,2}(\lambda)=0$ is easier to check since the $\hh_2(\lambda)$ and $\hh_4(\lambda)$ norms of a series can be computed in terms of the coefficients (see also Lemma~\ref{lemen} for a sufficient condition depending only on $\lambda$).
\end{remark}

We also mention that, in the particular case that $L(\lambda)=\infty$, Theorem~\ref{rem} ensures that $S_p(\lambda)$ is also $\infty$ for all $1\le p<\infty$, answering a question posed in \cite[Section~4.4]{CaDeMaSch21}.

As a consequence of \eqref{eqlip} we get a rigidity result for extremal cases.
Given a frequency $\lambda$, define
\begin{equation}   \label{defpo}
    p_{0}(\lambda):=\inf\{1\le p\le 2: \ S_p(\lambda)=L(\lambda)/2 \}.
\end{equation}
Since  $S_{p}(\lambda)$ is continuous as a function of $p$, the infimum is attained, that is, we have $S_{p_{0}(\lambda)}(\lambda)=L(\lambda)/2$.

\begin{proposition}\label{propext}
For any frequency $\lambda$, every $1\le p\le  p_0(\lambda)$ and $p\le q\le \infty$, we have
\begin{equation}    \label{eqext}
    S_p(\lambda)\le\tfrac{p_{0}(\lambda) L(\lambda)}{2p}, \,
     \text{ and } \,
    T_{q,p}(\lambda)\le \tfrac{p_{0}(\lambda) L(\lambda)}{2}(\tfrac{1}{p}-\tfrac{1}{q}).
\end{equation}
Moreover, assume that $p_{0}(\lambda)>1$. Given $1\le p< p_0$, the following are equivalent:
     \begin{enumerate}[label=(\roman*)]
	\item\label{it1} $S_p(\lambda)=\tfrac{p_{0}(\lambda) L(\lambda)}{2p}$;
    \item\label{it3} $S_r(\lambda)=\tfrac{p_{0}(\lambda) L(\lambda)}{2r}$ for every $p\le r \le p_{0}(\lambda)$;
    \item\label{it2} $T_{q,p}(\lambda)= \tfrac{p_{0}(\lambda) L(\lambda)}{2}(\tfrac{1}{p}-\tfrac{1}{q})$ for some $p< q \le \infty$;
        \item\label{it4} $T_{q,r}(\lambda)= \tfrac{p_{0}(\lambda) L(\lambda)}{2}(\tfrac{1}{r}-\tfrac{1}{q})$ for every $p\le r < p_{0}(\lambda)$ and $r< q \le \infty$.
    \end{enumerate}
    In particular, if $S_1(\lambda)=L(\lambda)$, then $p_{0}(\lambda)=2$  and all the strips involved are known.
\end{proposition}

Bayart's example from \cite[Theorem~5.2]{Ba} (mentioned 
after Remark~\ref{bilbao}) achieves the extremal behaviour from Proposition~\ref{propext} for $p_0(\lambda)=2$ showing that \eqref{rem1} and \eqref{eqlip} in Theorem~\ref{rem} cannot be improved. Next proposition extends this to all values of $p_0(\lambda)$.

\begin{proposition}\label{example}
  For any $1\le p_0 \le 2$, there exists a frequency $\lambda$ such that
  \begin{equation}       \label{eqpo}
      S_p(\lambda)=\begin{cases}
       \frac{p_0 L(\lambda)}{2p} & \text{if } 1\le p \le p_0
      \\  \frac{L(\lambda)}{2} &\text{if } p_0 \le p <\infty
  \end{cases}.
  \end{equation}
  In particular, for every $1\le p \le q \le 2$ there is a frequency $\lambda$ such that $S_p(\lambda)=L(\lambda)/q$.
\end{proposition}

We are not aware of any frequency that fails to satisfy \eqref{eqpo} for some $p_0$.
The proof of Proposition~\ref{example} uses a frequency $\lambda=\mu \cup \nu$, where $\mu$ is Bayart's example and $\nu$ is $\Q$-li, where the addition of $\nu$ serves to dilute the influence of the additive structure of $\mu$ in a controlled manner. The same frequency $\lambda$ is used for Proposition~\ref{propshift}.

Let us end this section by stating the main tool behind the results we present. It is a localization argument that allows to break down our analysis of the frequency $\lambda$ and study the behaviour of $\lambda \cap [j,j+1)$ for $j\in \N_0$ separately. This enables us to leverage more classical results related to $\Lambda(p)$-sets defined below and the theory of additive combinatorics. For more information on the subject, we refer for example to \cite{TaVu} .

For $0<p\le q \le \infty$ and a given finite or countable $A \subseteq [0,\infty)$, we denote by $\Lambda_{q,p}(A) \in [1, \infty]$ the infimum over all constants $C \ge 1$ so that
\begin{equation} \label{brandenburgo}
\Big\|\sum_{n\colon \lambda_n \in A} a_n e^{-\lambda_n s}\Big\|_q\leq C \Big\|\sum_{n\colon \lambda_n \in A} a_n e^{-\lambda_n s}\Big\|_p \,,
\end{equation}
for every choice of coefficients $a_n$. For a fixed $q$, we say that $A$ is a \emph{$\Lambda(q)$-set} (see~\cite{rudin1960trigonometric}) if $\Lambda_{q,p}(A) <\infty$ for some (and then every) $p<q$.

\begin{lemma}\label{theochar}
	For every $1\le p\le q\le \infty$ and every frequency $\lambda$,
	\[
	T_{q,p}(\lambda)=\limsup_{j\to\infty} \frac{\log \Lambda_{q,p}(\lambda\cap [j,j+1))}{j}\,.
	\]
\end{lemma}
Notice that for $p=2$, $q=2k$ (see for example \cite{LeLe}),
\begin{equation*}
	\Lambda_{2k,2}(A)\le \max_{x \in kA}\#\{x_1+\ldots +x_k=x\}^{1/2k}.
\end{equation*}
This recovers \cite[Proposition~5.4]{Ba} mentioned above.

\section{The examples}\label{secexa}
All known examples of hypercontractive frequencies can be recovered through our methods, albeit for some of them through a more involved argument. It is however worth going over them again through this more general point of view.

There are two ways for a frequency to be trivially hypercontractive, which are related to the analytic and the algebraic structure respectively: low density of points ($L(\lambda)=0$) and linear independence of $\lambda$ over $\Q$. The case $L(\lambda)=0$ follows from Theorem~\ref{corchar} and Proposition~\ref{cormax}, whereas the linearly independent case is a consequence of Corollary~\ref{burrito}. The interesting examples arise in the middle, when hypercontractivity is achieved despite
having non-trivial analytic and algebraic structures.

The two fundamental examples  $\lambda=(n)$ and $\lambda=(\log n)$ corresponding to the theory of Fourier and Dirichlet series respectively, are both hypercontractive. The first one is essentially due to the fact that 
{functions in Hardy spaces 
on the torus have Fourier series which
are  absolutely convergent on any smaller circle,}
whereas the hypercontractivity of $\lambda=(\log n)$ was proved by Bayart in \cite{Ba2} (see also
 \cite[Theorem~12.9]{ElLibro}). 
From our approach, the first one corresponds to the low density case. For the second one, in \cite[Example~5.8]{Ba} it is shown that the conditions of  Corollary~\ref{burrito} are satisfied. 
The same proof also works for frequencies of natural type $(R, (b_j)_j)$ with $(b_j)_j$ finite or $b_j\to \infty$, a fact originally established in \cite[Theorem~5.16]{DeFVScSP}. 
In Section~\ref{opti} we see that 
Bayart's example of a non-hypercontractive frequency  is of natural type, which shows that in the preceding result the additional assumptions on 
$(b_j)$ are not superfluous.

Let us also mention that, by \cite[Theorem~1.1]{DePe17},
\[
\Lambda_{q,p}(\{\log n \colon  n  \leq  x\}) = e^{\frac{\log x }{\log \log x}\big( \sqrt{\frac{p}{q}}+O
\big(\frac{\log \log \log x }{\log x}\big)  \big)}\,,
\]
which using  Lemma~\ref{theochar} leads to another proof for the hypercontractivity of  $(\log n)$.

The following example seems to be new and is motivated by the Hurwitz zeta function given by
\[
\zeta(s,\alpha)=\sum_{n=0}^\infty (n+\alpha)^{-s} \,,
\]
where say $0<\alpha\le 1$.
\begin{example} For any choice of $\alpha_1,\ldots,\alpha_d>0$, any frequency
\[
\lambda\subseteq \Big\{\log\Big(\sum_{j=1}^{d} \alpha_j m_j\Big) : \ m_1,\ldots,m_d \in \N_0\Big\}\,
\]
is hypercontractive. 

By \cite[Proposition~3]{Ch}, for every $k\in \N$ there is a constant $C>0$ depending on $d$ and $k$ such that for every $r,M>0$ we have
\[
  \#\Big\{(m_{ij})_{\genfrac{}{}{0pt}{}{1\le i \le k}{1\le j\le d}} \in \N_0^{k\times d} \colon  m_{ij}\le M \text{ and} \prod_{i=1}^k \Big(\sum_{j=1}^{d} \alpha_j m_{ij}\Big)=r \Big\}
  \le e^{C \frac{\log M}{\log \log M}}.
\]
Fix some $\mu>0$ and suppose that $\lambda_{n_1}+\ldots +\lambda_{n_k}=\mu$. Note that, for each $1\le i \le k$, we can find $m_{i1}, \ldots , m_{id} \in \mathbb{N}_{0}$ so that
\[
\lambda_{n_i}=\log\Big(\sum_{j=1}^d \alpha_j m_{ij}\Big) \,,
\]
and therefore,
\[
r:=e^\mu= \prod_{i=1}^k \Big(\sum_{j=1}^{d} \alpha_j m_{ij}\Big)\,.
\]
On the other hand, for each $1\le i \le k$ we clearly have  $\lambda_{n_i}\le \mu$. Then, for  every $i,j$ we must have
\[
0\le m_{i,j} \le \frac{e^\mu}{\alpha}=:M \,,
\]
where $\alpha=\min_j \alpha_j$. This altogether yields
\[
\log(\#\{ (n_1,\ldots,n_k) \colon   \lambda_{n_1}+\ldots+\lambda_{n_k} = \mu  \,       \} ) \le C \frac{\mu - \log \alpha}{\log (\mu - \log \alpha)}\,.
\]
By  Corollary~\ref{burrito}, we conclude that $\lambda$ is hypercontractive. \qed
\end{example}

It is worth mentioning that the previous example includes frequencies $\lambda=\log(\Lambda)$ such that $\Lambda$ is contained in a quasicrystal
(see for example \cite{La} for definitions and properties). More precisely, it is sufficient to assume $\Lambda$ is a Delone set  of finite type (which includes Meyer and cut-and-project sets).
For such a set it is easy to check that there is a finite number of possible distances $\alpha_1,\ldots,\alpha_d>0$ between two consecutive points in $\Lambda$, so we are in the conditions of the previous example.

As mentioned before, the first example of a non-hypercontractive frequency was given in
\cite[Theorem~5.2]{Ba}: a frequency $\mu$ such that $S_1(\mu)=L(\mu)$.
This leaves the question of whether there are frequencies $\lambda$  satisfying $S_1(\lambda)=L(\lambda)/2$ that are not hypercontractive. A probabilistic construction by Bourgain \cite[Theorem~2]{Bo} (see also \cite[Theorem~4.8]{rudin1960trigonometric} for the even-integer case) shows that for $2<p<\infty$ there are $\Lambda(p)$-sets that are not $\Lambda(q)$-sets for any $q>p$. This provides a scale that distinguishes sets with low additive structure, which readily translates to our setting and, in particular, sets the concept of hypercontractivity and  the equality $S_1(\lambda)=L(\lambda)/2$ apart.

\begin{example}
For $2<p<\infty$, there exists a frequency $\lambda$, such that  $T_{p,2}(\lambda)=0$, but $T_{q,2}(\lambda)>0$ for every $q>p$. In particular, there is a non-hypercontractive frequency $\lambda$ such that $S_1(\lambda)=L(\lambda)/2$.

Fix $2<p<q<\infty$. From \cite[Theorem~1]{Bo} (see also the proof of \cite[Theorem~2]{Bo}), for each $j \in \mathbb{N}$  we can construct a set $A_j\subseteq \{1,\ldots,2^j-1\}$ such that:
\begin{itemize}
		\item $\# A_j=\lfloor 4^{j/p} \rfloor$;
		\item there is a constant $C=C(p)\ge 1$ such that for every choice of coefficients $(a_n)_{n\in A_j}\subseteq \C$ we have
		\begin{equation} \label{freude}
		\Big(\int_\T \Big|\sum_{n\in A_j}a_n z^n \Big|^p dz \Big)^{1/p}
		\le C \Big( \sum_{n \in A_{j}} \vert a_n \vert^{2} \Big)^{1/2} = C \Big(\int_\T \Big|\sum_{n\in A_j}a_n z^n \Big|^2 dz \Big)^{1/2} \,;
		\end{equation}
		\item there is an absolute constant $c>0$ such that
		\begin{equation} \label{freiheit}
            \begin{split}
                \Big(\int_\T \Big|\sum_{n\in A_j}a_n z^n \Big|^q dz \Big)^{1/q}\ge & c 2^{j(1/p-1/q)}  (\# A_j)^{1/2} \\ &
                = c 2^{j(1/p-1/q)}\Big(\int_\T \Big|\sum_{n\in A_j}a_n z^n \Big|^2 dz \Big)^{1/2}\,.
            \end{split}
		\end{equation}
\end{itemize}
We define now a frequency $\lambda$ as
\[
\lambda = \big\{ j + 2^{-j} n \colon j \in \mathbb{N}, \, n \in A_{j} \big\}
\]
so that, for each $j \in \mathbb{N}$ we have
\[
\lambda \cap [j,j+1) = j+ 2^{-j}A_j\,.
\]

Note now that for each $j$, each choice of coefficients $(a_n)_{n \in A_j}$, and each $1 \leq r < \infty$ we have
\begin{align*}
\Big\Vert \sum_{n \in A_j} a_ne^{-i(j +2^{-j} n) s}\Big\Vert_r
&
=
\lim_{T \to \infty } \Big(\frac{1}{2T} 
\int_{-T}^{T }\big| \sum_{n \in A_j} a_ne^{-i(j +2^{-j} n)t}\big|^rdt\Big)^{1/r}
\\&
=
\lim_{T \to \infty } \Big(\frac{1}{2T}
\int_{-T}^{T}\big| \sum_{n \in A_j} a_n e^{-i2^{-j} n t}\big|^rdt\Big)^{1/r}
\\&
=
\lim_{T \to \infty } \Big(\frac{1}{2T2^{-j}}
\int_{-T 2^{-j}}^{T 2^{-j}}\big| \sum_{n \in A_j} a_n e^{-i n t}\big|^r dt\Big)^{1/r}
=\Big\Vert \sum_{n \in A_j} a_n z^n \Big\Vert_r\,.
\end{align*}

Then, as a straightforward consequence of \eqref{freude} and \eqref{freiheit} 
we get
\[
\Lambda_{p,2}(\lambda \cap [j,j+1)) \le C, \quad \text{and,} \quad \Lambda_{q,2}(\lambda \cap [j,j+1))\ge c 2^{j(1/p-1/q)} \,.
\]
Applying Lemma~\ref{theochar} we have
 \[
T_{p,2}(\lambda) =0, \quad \text{and,} \quad T_{q,2}(\lambda)\ge \Big(\frac{1}{p}-\frac{1}{q}\Big)\log 2 \,.
\]
 In particular, $\lambda$ is not hypercontractive but by Remark~\ref{bilbao-2} (see also Corollary \ref{corlil}) we have $S_1(\lambda)=L(\lambda)/2$. \qed
\end{example}

\section{The proofs} \label{sec-proofs}

We address now the proofs of the results stated in the introduction. We structure this in several blocks (or subsections). In the first one we give some preliminary results about $L(\lambda)$ and the behaviour of $\Lambda_{q,p}$ that will be needed later. The second block focuses on localization results. We prove first our main technical tool (Lemma~\ref{theochar}), and then Proposition~\ref{cormax}. In the third subsection we answer the question of when is a frequency hypercontractive, giving the proofs of Theorem~\ref{corchar} and Corollary~\ref{burrito}. In the fourth block we look at the width of the $\mathcal{H}_{p}$-strips and prove Theorem~\ref{rem} and Proposition~\ref{propext}.   We finish this note by analyzing the optimality of the bounds that we have obtained so far, both for localization (in Proposition~\ref{propshift}) and for strips (in Proposition~\ref{example}).

\subsection{Preliminaries}\label{secprel}
As described in the introduction the Hardy  space $\mathcal{H}_p(\lambda)$ of $\lambda$-Dirichlet series is the completion of the linear space of all
$\lambda$-Dirichlet polynomials under the $p$-norm defined in~\eqref{pnorm}. At a few occasions we need an alternative approach through groups. We give a brief overview of this point of view and refer to \cite{DeSch} for a detailed exposition.

Let $G$ be a compact abelian group and $\beta_G: (\mathbb{R},+)\to  G$
a continuous homomorphism with dense range. Then the pair $(G,\beta_G)$ is said to be a $\lambda$-Dirichlet group if there is a sequence
$(h_{\lambda_n})$ of characters in $\widehat{G}$ (the dual group of $G$) such that $h_{\lambda_n} \circ \beta_G = e^{-\lambda_n \boldsymbol{\cdot}}$ for all $n$.
Then for each $1 \leq p \leq \infty$, the  Banach space $H_p^\lambda(G)$ is defined to be the subspace of all $f \in L_p(G)$
(with respect to the Haar measure on $G$) for which $\supp \widehat{f} \subseteq \{ h_{\lambda_n}\colon n\in \mathbb{N} \}$. In the following we collect a few results from~\cite{DeSch}, which are relevant for our purposes.

\begin{itemize}
  \item
  A $\lambda$-Dirichlet series $D= \sum a_n e^{-\lambda_n s}$ belongs to $\mathcal{H}_p(\lambda)$ if and only if there is (a unique) $f \in H_p^\lambda(G)$ such that
  $a_n = \widehat{f}(h_{\lambda_n})$ for all $n$, and in this case $\|D\|_p = \|f\|_p$. In other terms, the mapping\
  $\mathcal{H}_p(\lambda)\rightarrow H_p^\lambda(G)$ given by $D \mapsto f$
  identifies  both spaces as Banach spaces (we say that $D$ and $f$ are associated to each other).
  \item
  Given a frequency $\lambda$, there always exists a $\lambda$-Dirichlet group $(G,\beta_G)$, and the  Hardy space $H_p^\lambda(G)$ does not depend on the chosen $\lambda$-Dirichlet group.
  \item
  Most frequencies $\lambda$ generate a somewhat natural $\lambda$-Dirichlet group $(G,\beta_G)$. For example, for $\lambda=(n)$ 
  the circle group $\mathbb{T}$ and $\beta_G(t)= e^{-int}$ form a  $\lambda$-Dirichlet group, and for $\lambda=(\log n)$ the infinite dimensional torus $\mathbb{T}^\infty$ combined with the Kronecker flow
  $\beta_G(t)= (p_k^{-it})_k$  create such a group (here $p_k$ stands for the $k$th prime).
  \item
  Let  $(G,\beta_G)$ be a $\lambda$-Dirichlet group and $\sigma >0$. Then the push forward measure 
  under $\beta_G$
  of the Poisson kernel 
  $P_\sigma(t) = \frac{\sigma}{\pi(t^2 + \sigma^2)}$ on $\mathbb{R}$ defines the Poisson measure $p_\sigma$ on $G$, and we have that $\tau_\sigma D$ and $p_\sigma \ast f$ are associated to each other whenever
   $D\in  \mathcal{H}_p(\lambda)$ and $f\in H_p^\lambda(G)$ are. 
   
   Consequently, every result on the   translation operator $\tau_\sigma$ between Hardy spaces of $\lambda$-Dirichlet series transfers into a result on the convolution operator $p_\sigma \ast \boldsymbol{\cdot}$ between Hardy spaces on $\lambda$-Dirichlet groups, and vice versa.
 \end{itemize}

Note that in the definition of $\mathcal{H}_p(\lambda)$ given in the introduction, we avoid the case  $p= \infty$. 
We handle this case now. The Banach space $\mathcal{H}_\infty(\lambda)$ consists of all Dirichlet series 
$D = \sum \widehat{f}(h_{\lambda_n}) e^{-\lambda_n s} $, where $f \in H_\infty^\lambda(G)$ and  $(G,\beta_G)$ is (any) $\lambda$-Dirichlet group.
 Clearly, the norm is given  by $\|D\|_\infty = \|f\|_\infty$. It is important here the fact that  for any Dirichlet polynomial $D = \sum a_n e^{-\lambda_n s} $
\[
\|D\|_\infty = \sup_{t \in \mathbb{R}} \big| \sum a_n e^{-i\lambda_n t} \big|\,.
\]

\vspace{2ex}

As a very first consequence of this group approach to Hardy spaces of Dirichlet series, we immediately see that 
\begin{equation}    \label{eqcoef}
    |a_n|=|\widehat{f}(h_{\lambda_n})|\le \|f\|_p=\|D\|_p.
\end{equation}
for every $D=\sum a_n e^{-\lambda_n s}\in \mathcal{H}_p(\lambda)$ and~$n$.
Even more, given $D=\sum a_n e^{-\lambda_n s}\in \mathcal{H}_p(\lambda)$ we can apply the Hausdorff-Young inequality for  compact abelian groups  to $G$ and $\widehat{G}$ (see e.g. \cite[Theorem~9.5.2]{Garling-book}) to have 
\begin{equation}\label{hy1}
    \Big(\sum |a_n|^{p'}\Big)^{1/p'} \le \|D\|_p\,, \,
    \text{ if } 1\le p\le 2\,,
\end{equation}
    and
\begin{equation}   \label{hy2}
    \|D\|_p \le \Big(\sum |a_n|^{p'}\Big)^{1/p'}\,,\, \text{ if  } 2\le p\le \infty\,.
\end{equation}

By Littlewood's inequality \cite[Theorem~5.5.1]{Garling-book} (also known as Lyapunov's inequality, see \cite[Lemma~II.4.1]{Werner-book})  for $0< p < r \le q\le\infty$ and any 
$f\in L^{p}(\mu)\cap L^{q}(\mu)$,
\begin{equation}\label{Little}
    \|f\|_r\le \|f\|_p^\alpha \|f\|_q^{1-\alpha},
\end{equation}
where
\[
\alpha=\frac{\tfrac{1}{r}-\tfrac{1}{q}}{\tfrac{1}{p}-\tfrac{1}{q}}\,.
\]

From this we get the following Nikolskii-type inequality which is a formal extension of \cite[Theorem~1]{NeWi} (see also \cite[Theorem~2.6]{DeVore-Lorentz-book}), where the result is given for classical trigonometric polynomials on $\mathbb{R}^n$. Although the proof given in \cite{NeWi}  works line by line for our more  general situation,  we prefer to give it for the sake of completeness.

\begin{lemma} \label{N-type}
  Let $G$ be a unimodular  locally compact  group with Haar measure $m$, and  $1 \leq p \leq 2,\, p \leq q \leq  \infty$.
  Then for every trigonometric polynomial
  $f = \sum_{\gamma \in  A} \widehat{f}(\gamma) \gamma$, where $A\subset \widehat{G}$ is finite (here $\widehat{G}$ is the dual group of $G$), we have
  \[
  \|f\|_{q} \leq (\# A)^{\frac{1}{p}-\frac{1}{q}} \|f\|_{p}\,.
  \]
\end{lemma}
\begin{proof}
Define $h = \sum_{\gamma \in A} \gamma$. By Young's  convolution inequality and the Hausdorff-Young inequality for groups we have
  \[
   \|f\|_\infty = \|h\ast f\|_\infty \leq \|h\|_{p'}\|f\|_p   \leq  (\# A)^{\frac{1}{p}} \|f\|_p\,.
  \]
 Consequently, by Littlewood's inequality \eqref{Little} for $1 \leq p \leq 2, p \leq q \leq  \infty$
  \[
  \|f\|_q \leq \|f\|_p^{\frac{p}{q}} \|f\|_\infty^{1-\frac{p}{q}}
  \leq
  \|f\|_p^{\frac{p}{q}}
    (\# A)^{\frac{1}{p}-\frac{1}{q}} \|f\|_p^{1-\frac{p}{q}}\,. \qedhere
    \]
\end{proof}

We remark that the group approach to Hardy spaces series readily gives reformulations of \eqref{Little} or Lemma \ref{N-type} in terms of Dirichlet series. Littlewood's inequality also allows us to compare $\Lambda_{q,p}$ constants. The following lemma is essentially \cite[Theorem~1.4]{rudin1960trigonometric}. We include a proof because we need a slight generalization.
\begin{lemma}\label{lelil}
    For a finite set $A\subseteq \R$, $0< p_1\le p_2 < q_2\le \infty$ and $p_1\le q_1 \le q_2$, we have
    \[
    \Lambda_{q_1,p_1}(A)\le \Lambda_{q_2,p_2}(A)^{\alpha}\,,
    \]
    where
    \[
    \alpha=\frac{\tfrac{1}{p_1}-\tfrac{1}{q_1}}{\tfrac{1}{p_2}-\tfrac{1}{q_2}}\,.
    \]
\end{lemma}
\begin{proof}
    Let $D$ be a Dirichlet polynomial with frequencies in $A$. By Littlewood's inequality  \eqref{Little},
    \[
    \|D\|_{q_1}\le \|D\|_{p_1}^{\alpha_1} \|D\|_{q_2}^{1-\alpha_1}\le \|D\|_{p_1}^{\alpha_1} \Lambda_{q_2,p_2}(A)^{1-\alpha_1} \|D\|_{p_2}^{1-\alpha_1}\,,
    \]
    where
    \[
    \alpha_1=\frac{\tfrac{1}{q_1}-\tfrac{1}{q_2}}{\tfrac{1}{p_1}-\tfrac{1}{q_2}}\,.
    \]
    Similarly,
    \[
    \|D\|_{p_2}\le \|D\|_{p_1}^{\alpha_2} \|D\|_{q_2}^{1-\alpha_2}\le \|D\|_{p_1}^{\alpha_2} \Lambda_{q_2,p_2}(A)^{1-\alpha_2} \|D\|_{p_2}^{1-\alpha_2}\,,
    \]
    where
    \[
    \alpha_2=\frac{\tfrac{1}{p_2}-\tfrac{1}{q_2}}{\tfrac{1}{p_1}-\tfrac{1}{q_2}}\,.
    \]
    Rearranging, we get
    \[
    \|D\|_{p_2}\le \Lambda_{q_2,p_2}(A)^{1/\alpha_2-1}  \|D\|_{p_1} \,.
    \]
    So,
    \[
    \|D\|_{q_1}\le  \|D\|_{p_1} \Lambda_{q_2,p_2}(A)^{1-\alpha_1} \Lambda_{q_2,p_2}(A)^{(1/\alpha_2-1)(1-\alpha_1)}= \Lambda_{q_2,p_2}(A)^{(1-\alpha_1)/\alpha_2}\|D\|_{p_1}\,.
    \]
    Expanding the exponent $(1-\alpha_1)/\alpha_2$ we arrive at the result.
\end{proof}

Next, we recall Helson's bound for the partial sum operator. For each $N \in \mathbb{N}$, we write $\pi_{N}$ for the $N$-th partial sum operator given by
\[
\pi_{N} \big( \sum_{n=1}^{\infty} a_{n} e^{-\lambda_{n} s} \big) =  \sum_{n=1}^{N} a_{n} e^{-\lambda_{n} s} \,.
\]
Using the group approach that we presented in Section~\ref{secprel} to define the spaces $\hh_p(\lambda)$ for $p \geq 1$ we may as well define such spaces for $0<p < 1$. With this, in \cite{He} (see also \cite[Theorem~8.7.6]{Ru}), the following is shown.

\begin{lemma}\label{parsum}
There exists a constant $C\ge 1$ such that for every frequency $\lambda$, and all $0 < \varepsilon <1$ 
and $N$ we have
\[
\|\pi_N\|_{\hh_{1-\varepsilon}(\lambda)\to\hh_1(\lambda)}\le \frac{C}{\varepsilon}\,.
\]
\end{lemma}

Finally, we need another characterization of $L(\lambda)$ as a sort of upper asymptotic density.
\begin{lemma}
	\label{lella}
	For every frequency $\lambda$, we have
	\[
	L(\lambda)=\limsup_{j\to\infty} \frac{\log \#(\lambda\cap [j,j+1))}{j}\,.
	\]
\end{lemma}
\begin{proof}
	Notice that for every $\sigma\in \R$, the following holds
	\[
	e^{-\sigma} \sum_{j\ge 0} \#(\lambda\cap [j,j+1)) e^{-\sigma j}
	\le \sum_{n\ge 1}  e^{-\sigma \lambda_n} \le \sum_{j\ge 0} \#(\lambda\cap [j,j+1)) e^{-\sigma j}\,.
	\]
	So, the values $\sigma$ ensuring convergence coincide for both series. In particular, the critical value $L(\lambda)$ for the convergence of the middle series can be computed by the $n$-th root convergence test applied to the other series:
	\[
	\limsup_{j\to\infty} \Big(\#(\lambda\cap [j,j+1)) e^{-L(\lambda) j}\Big)^{1/j}=1\,.
	\]
	Taking $\log$ and rearranging the expression completes the proof.	
\end{proof}

\subsection{Localization}

In this section we prove the localization result Lemma~\ref{theochar},  and Proposition~\ref{cormax}, that gives conditions on $\sigma$ so that $\tau_{\sigma} : \mathcal{H}_{p} \to \mathcal{H}_{q}$ is well defined. To begin with, we need the following lemma.

\begin{lemma}\label{monotone2}
	  Let $1\le p < \infty$. For every frequency $\lambda$, every Dirichlet polynomial $D=\sum_{n=N}^{M} a_n e^{-\lambda_n s}$ and every $\sigma>0$, we have
	\begin{equation}\label{eqequi2}
		e^{-\lambda_M \sigma} \|D\|_p\le \|D_\sigma\|_p \le  e^{-\lambda_N \sigma}\|D\|_p.
	\end{equation}
\end{lemma}
\begin{proof}
Consider the frequency $\widetilde{\lambda}=\{\widetilde{\lambda}_n\}_{n=N}^M$ given by
\[
\widetilde{\lambda}_n=\lambda_n-\lambda_N\,,
\]
and define 
\[
\widetilde{D}=\sum_{n=N}^{M} a_n e^{-\widetilde{\lambda}_n s}\,.
\]
From \eqref{pnorm} and the fact that $\tau_\sigma$ is a contraction in $\hh_p(\widetilde{\lambda})$ we see that
\[
\|D_\sigma\|_p=e^{-\lambda_N \sigma}\|\widetilde{D}_\sigma\|_p
\le e^{-\lambda_N \sigma}\|\widetilde{D}\|_p=e^{-\lambda_N \sigma}\|D\|_p\,.
\]
It remains to check the left-hand side of \eqref{eqequi2}. Similarly as before, consider $\widetilde{\lambda}=\{\widetilde{\lambda}_n\}_{n=N}^M$ given by
\[
\widetilde{\lambda}_n=\lambda_M-\lambda_n\,,
\]
and define 
\[
\widetilde{D}=\sum_{n=N}^{M} a_n e^{-\lambda_n \sigma} e^{-\widetilde{\lambda}_n s}\,.
\]
From \eqref{pnorm} and the fact that $\tau_\sigma$ is a contraction in $\hh_p(\widetilde{\lambda})$ we conclude that
\[
\|D\|_p=e^{\lambda_M \sigma}\|\widetilde{D}_\sigma\|_p
\le e^{\lambda_M \sigma}\|\widetilde{D}\|_p=e^{\lambda_M \sigma}\|D_\sigma\|_p\,. \qedhere
\]
\end{proof}

We are now ready to prove our main tool Lemma~\ref{theochar}.

\begin{proof}[Proof of Lemma~\ref{theochar}] 
Given $1\le p < q\le \infty$,
	first assume that $\tau_\sigma$ maps $\hh_p(\lambda)$ into $\hh_q(\lambda)$. Fix  $j\ge 0$ and a Dirichlet polynomial $D=\sum_{n: \lambda_{n}\in [j,j+1) } a_n e^{-\lambda_{n} s}$. Notice that
	\begin{align*}
		\|D\|_q & = \|\tau_\sigma(D_{-\sigma})\|_q
		\le  \|\tau_\sigma\| \|D_{-\sigma}\|_p
		\le  \|\tau_\sigma\| \|D\|_{p} e^{(j+1)\sigma},
	\end{align*}
	where in the last step we used Lemma~\ref{monotone2}.\\
So (recall \eqref{brandenburgo} for the definition of $\Lambda_{q,p})$,
	\[
		\limsup_{j\to\infty}\frac{\log \Lambda_{q,p}(\lambda\cap [j,j+1))}{j} \le \limsup_{j\to\infty}\frac{\log(  \|\tau_\sigma\| e^{(j+1)\sigma})}{j} = \sigma.
	\]
Therefore,
\[
\limsup_{j\to\infty} \frac{\log \Lambda_{q,p}(\lambda\cap [j,j+1))}{j} \le  T_{q,p}(\lambda) \,.
\]

For the reverse inequality, write
\[
\sigma= \limsup_{j\to\infty} \frac{\log \Lambda_{q,p}(\lambda\cap [j,j+1))}{j}\,.
\]
For $\varepsilon>0$, there exists $j_0\in \N$ such that 
\begin{equation}\label{eqexp}
	\Lambda_{q,p}(\lambda\cap [j,j+1))\le e^{j(\sigma+\varepsilon)}
\end{equation}
for every $j\ge j_0$. Now, for every such $j$ Lemma~\ref{lelil} yields
\begin{equation}\label{eqeps}
	\Lambda_{p,p-\varepsilon}(\lambda\cap [j,j+1))\le \Lambda_{q,p}(\lambda\cap [j,j+1))^{\alpha} \le e^{j(\sigma+\varepsilon)\alpha},
\end{equation}
where
\[
\alpha=\frac{\tfrac{1}{p-\varepsilon}-\tfrac{1}{p}}{\tfrac{1}{p}-\tfrac{1}{q}}\,.
\]
For $p>1$ there is no need to work with $p-\varepsilon$ nor \eqref{eqeps}. We do it to simultaneously cover the case $p=1$. 

Let $D=\sum a_n e^{-\lambda_{n} s}\in \hh_p(\lambda)$. By the triangle inequality, observe that for $\tilde{\sigma}=(\sigma+2\varepsilon)(1+\alpha)$,
\begin{equation}\label{eqtri}
	\|\tau_{\tilde{\sigma}} D\|_q \le  \sum_{n: \lambda_{n}< j_0} \Vert a_n e^{-\lambda_{n} s} \Vert_{q} +  \sum_{j\ge j_0} \|\sum_{n: \lambda_{n}\in [j,j+1) } a_n e^{-\lambda_{n} \tilde{\sigma}} e^{-\lambda_{n} s}\|_q.
\end{equation}
Using \eqref{eqcoef}, the first term on the right-hand side can be bounded by (recall that $\Vert a_{n} e^{-\lambda_{n} s} \Vert_{q} = \vert a_{n} \vert$)
\begin{equation}\label{eqbrut}
	\sum_{n: \lambda_{n}< j_0}|a_n| \le \#(\lambda\cap [0,j_0)) \|D\|_p.
\end{equation}
Regarding the second term, from \eqref{eqexp}, Lemma~\ref{monotone2} and \eqref{eqeps} we get
\begin{align*}
	\sum_{j\ge j_0} \|\sum_{n: \lambda_{n}\in [j,j+1) }& a_n e^{-\lambda_{n} \tilde{\sigma}} e^{-\lambda_{n} s}\|_q
	\\&\le \sum_{j\ge j_0} \Lambda_{q,p}(\lambda\cap [j,j+1)) \|\sum_{n: \lambda_{n}\in [j,j+1) } a_n e^{-\lambda_{n} \tilde{\sigma}} e^{-\lambda_{n} s}\|_{p}
	\\&\le  \sum_{j\ge j_0} e^{j(\sigma+\varepsilon)} e^{-j \tilde{\sigma}} \|\sum_{n: \lambda_{n}\in [j,j+1) } a_n  e^{-\lambda_{n} s}\|_{p}
	\\&\le  \sum_{j\ge j_0} e^{j(\sigma+\varepsilon)(1+\alpha)} e^{-j \tilde{\sigma}} \|\sum_{n: \lambda_{n}\in [j,j+1) } a_n  e^{-\lambda_{n} s}\|_{p-\varepsilon}
	\\&\le C_{p,\varepsilon} \|D\|_p \sum_{j\ge j_0} e^{-j\varepsilon(1+\alpha)}\,,
\end{align*}
where in the last step we used Lemma~\ref{parsum} if $p=1$, whereas for $p>1$ we use that the monomials form a Schauder basis (see \cite[Theorem~4.16]{DeSch}). As mentioned before, notice that the use of \eqref{eqeps} is redundant whenever $p>1$ since we can estimate the $p$-norms directly.

Joining this with \eqref{eqtri} and \eqref{eqbrut} we deduce that $\tau_{\tilde\sigma}:\hh_p\to\hh_q$ is bounded. Since $\varepsilon$ was arbitrary and $\alpha \to 0$ as $\varepsilon \to 0$,we conclude that
\[
T_{q,p}(\lambda)\le \limsup_{j\to\infty} \frac{\log \Lambda_{q,p}(\lambda\cap [j,j+1))}{j}.\qedhere
\]
\end{proof}

\begin{remark}
	The partition $[j,j+1)$ in Lemma~\ref{theochar} can be replaced by $\delta[j,j+1)+x$ for fixed $\delta,x$. Moreover, one can even take $[r_j,r_{j+1})$ and replace the denominator $j$ with $r_{j}$, as long as
\[
\lim_{j\to \infty}\frac{r_{j+1}}{r_j}=1, \, \text{ and } \, \lim_{j\to \infty}\frac{r_j}{\log j}=\infty \,.
\]
\end{remark}

As a direct consequence of localization and Lemma~\ref{lelil} we get the following result.
\begin{corollary}\label{corlil}
For every frequency $\lambda$ and $0<p_1\le \, p_2,q_1\, \le q_2\le \infty$, we have
    \[
    \big(\frac{1}{p_2}-\frac{1}{q_2} \big) T_{q_1,p_1}(\lambda) \le \big( \frac{1}{p_1}-\frac{1}{q_1} \big)T_{q_2,p_2}(\lambda) \,.
    \]
\end{corollary}

Similarly, Proposition~\ref{cormax} also follows from localization and the behaviour of the constant $\Lambda_{q,p}$ .

\begin{proof}[Proof of Proposition~\ref{cormax}]
We give two proofs.
By Lemma~\ref{N-type}
we have
\[
\Lambda_{q,p}(\lambda\cap [j,j+1))\le \#(\lambda\cap [j,j+1))^{1/p-1/q} \,.
\]
So,
\[
\limsup_{j\to\infty} \frac{\log \Lambda_{q,p}(\lambda\cap [j,j+1))}{j}
\le \Big(\frac{1}{p}-\frac{1}{q}\Big)\limsup_{j\to\infty} \frac{\log \#(\lambda\cap [j,j+1))}{j}= \Big(\frac{1}{p}-\frac{1}{q}\Big)L(\lambda) \,,
\]
where in the last step we used Lemma~\ref{lella}. The conclusion now follows from Lemma~\ref{theochar}.

As an alternative proof, from Corollary~\ref{corlil}, \eqref{izagirre} and assuming the monotonicity of $p S_p(\lambda)$ proven (see Theorem~\ref{rem}), we see that
\begin{multline*}
    T_{q,p}(\lambda)   \le  \Big(\frac{1}{p}-\frac{1}{q}\Big)p T_{\infty,p}(\lambda)\le  \Big(\frac{1}{p}-\frac{1}{q}\Big)p S_{p}(\lambda) \\ 
    \le  \Big(\frac{1}{p}-\frac{1}{q}\Big) 2 S_{2}(\lambda) =  \Big(\frac{1}{p}-\frac{1}{q}\Big)L(\lambda). \qedhere
\end{multline*}
\end{proof}

\smallskip

\subsection{Hypercontractivity}\label{sechyper}
We study  sufficient  conditions  which  guarantee hypercontrativity of arbitrary  frequencies. The main aim  is to prove Theorem~\ref{corchar}, which then implies Corollary~\ref{burrito} (as explained in the introduction).

We start with the following simple reformulation of the $k$-th additive energy of a finite set $A\subseteq \R$:  
 \begin{equation} \label{corona}
    \begin{split}
       \big\| \sum_{n: \lambda_n\in A} & e^{-\lambda_n s} \big\|_{2k}^{2k}
        = \big\| \big(\sum_{n: \lambda_n\in A} e^{-\lambda_n s}\big)^k \big\|_{2}^{2}
        \\&
                = \big\| \sum_{l: \mu_l\in M_k(A)} \#\{(\lambda_{n_1},\ldots,\lambda_{n_k})\in A^k: \lambda_{n_1} +\ldots +\lambda_{n_k}=\mu_l\} e^{-\mu_l s}\big\|_{2}^{2}
        \\
        &
        = \sum_{l: \mu_l\in M_k(A)} \#\{(\lambda_{n_1},\ldots,\lambda_{n_k})\in A^k: \lambda_{n_1} +\ldots +\lambda_{n_k}=\mu_l\}^2=E_k(A)\,,
    \end{split}
 \end{equation}
  where    
    $M_k(A)=\left\{x\in\R: x=x_1+ \cdots +x_k,\ x_i\in A\right\}$.
    
    \smallskip

        The following lemma links Theorem~\ref{corchar} with Lemma~\ref{theochar}.

 \begin{lemma}\label{lemenI}
 There exists a constant $C>0$, such that for any finite set $A\subseteq \R$ and any $k\in \N$ the following holds
 \[
 \sup_{A'\subseteq A} \frac{E_k(A')^{1/2k}}{\sqrt{\#A'}}
 \leq
   \Lambda_{2k,2}(A)\le C \sqrt{\log \big(\#A+2\big)}  \sup_{A'\subseteq A} \frac{E_k(A')^{1/2k}}{\sqrt{\#A'}}.
   \]
 \end{lemma}

\begin{proof}
To begin with, note that for each subset  $A'\subseteq A$, \eqref{corona} yields
\[
\frac{E_k(A')^{1/k}}{\#A'}= \frac{\big\| \sum_{n: \lambda_n\in A'} e^{-\lambda_n s} \big\|_{2k}^2}{\big\| \sum_{n: \lambda_n\in A'} e^{-\lambda_n s} \big\|_{2}^2}\le \Lambda_{2k,2}(A)^2 \,,
\]
which immediately gives the lower bound.

Conversely, let $A=\{\mu_n\}_{n=1}^N\subseteq \R$ be a set of $N$ elements and let $D=\sum_{n=1}^N a_n e^{-\mu_n s}$ be a Dirichlet polynomial. Take $\theta$, a permutation of $\{1,\ldots, N\}$ such that $|a_{\theta(l)}|$ is non-decreasing and define $l_0\ge 1$ to be the smallest integer such that
	\[
	\sum_{l=1}^{l_0} |a_{\theta(l)}|\ge \|D\|_2 \,.
	\]
	Now partition $\{l_0,\ldots, N\}$ into (possibly empty) sets $\{\Gamma_r\}_{r=r_1}^{r_2}$ with $r_1,r_2\in\zz$ such that $2^r\le |a_{\theta(l)}|<2^{r+1}$, whenever $l\in \Gamma_r$ and $r_1 \leq r \leq r_2$. Notice that $|a_{\theta(N)}|\le \|D\|_2  \le  l_0 |a_{\theta(l_0)}|$. So, we can choose $r_1$ and $r_2$ such that  $r_2-r_1 \leq C \log(N+2)$, where $C>0$ is an absolute constant. Combining all this with \eqref{corona} gives that  
\begin{align*}
		\|D\|_{2k}&= \Big\| \sum_{l=1}^{l_0-1} a_{\theta(l)} e^{-\mu_{\theta(l)} s} + \sum_{r=r_1}^{r_2} \sum_{l\in \Gamma_r} a_{\theta(l)} e^{-\mu_{\theta(l)} s} \Big\|_{2k}
		\\ &\le  \sum_{l=1}^{l_0-1} |a_{\theta(l)}| +  \sum_{r=r_1}^{r_2} \Big\| \sum_{l\in \Gamma_r} a_{\theta(l)} e^{-\mu_{\theta(l)} s} \Big\|_{2k}
		\\ &= \|D\|_2 +  \sum_{r=r_1}^{r_2} \Big\| \sum_{l_1,\ldots,l_k\in \Gamma_r} \prod_{i=1}^k a_{\theta(l_i)} e^{-\mu_{\theta(l_i)} s}\Big\|_2^{1/k}
		\\ &= \|D\|_2 +  \sum_{r=r_1}^{r_2} \Big(
\sum_{\nu \in \{ \sum_{i=1}^k \mu_{\theta(l_i)} \colon l_1,\ldots l_k  \in \Gamma_r\}}
\Big
(\sum_{\stackrel{l_1,\ldots,l_k\in \Gamma_r}{\sum \mu_{\theta(l_i)}=\nu}} \prod_{i=1}^k |a_{\theta(l_i)}| \Big)^2 \Big)^{1/2k}
		\\ &\le \|D\|_2 +  \sum_{r=r_1}^{r_2} 2^{r+1}\Big(\sum_{\nu \in \{ \sum_{i=1}^k \mu_{\theta(l_i)} \colon l_1,\ldots l_k  \in \Gamma_r\}} \Big(\sum_{\stackrel{l_1,\ldots,l_k\in \Gamma_r}{\sum \mu_{\theta(l_i)}=\nu}} 1 \Big)^2 \Big)^{1/2k}
		\\&=\|D\|_2 +   \sum_{r=r_1}^{r_2} 2^{r+1}E_k(\{\mu_{\theta(l)}: l\in \Gamma_r\})^{1/2k}
		\\ &\le \|D\|_2 +   2 \sum_{r=r_1}^{r_2} 2^r \sqrt{\#\Gamma_r} \sup_{A'\subseteq A} \frac{E_k(A')^{1/2k}}{\sqrt{\#A'}}
		\\ &\le \|D\|_2 +   2 \sum_{r=r_1}^{r_2} (\sum_{l\in \Gamma_r} |a_{\theta(l)}|^2)^{1/2}\sup_{A'\subseteq A} \frac{E_k(A')^{1/2k}}{\sqrt{\#A'}}
		\\ &\le C \sqrt{ \log (N+2)} \|D\|_2 \sup_{A'\subseteq A} \frac{E_k(A')^{1/2k}}{\sqrt{\#A'}},
	\end{align*}
	 where in the last step we used the Cauchy-Schwarz inequality and the constant $C$ may have changed but remains universal. This shows that
 \begin{equation*}
   \Lambda_{2k,2}(A)\le C \sqrt{ \log \big(\#A+2\big)}  \sup_{A'\subseteq A} \frac{E_k(A')^{1/2k}}{\sqrt{\#A'}}. \qedhere
 \end{equation*}
		    \end{proof}

\smallskip

\begin{lemma}	\label{lemen}
Let $\lambda$ be a frequency. Then for every integer $k\ge 2$ we have
\[
\limsup_{j\to\infty} \sup_{A\subseteq \lambda\cap [j,j+1)}\frac{1}{2j}\log \Big(\frac{E_k(A)^{1/k}}{\#A}\Big)
\leq T_{2k,2}(\lambda)\,.
\]
Moreover, if  $\limsup \frac{\log\log n}{\lambda_n}=0$, then this is an equality.
\end{lemma}

\begin{proof}
    	 By  (the lower bound in) Lemma~\ref{lemenI} we have
    \begin{equation*}
      \limsup_{j\to\infty} \sup_{A\subseteq \lambda\cap [j,j+1)}\frac{1}{2j}\log \Big(\frac{E_k(A)^{1/k}}{\#A}\Big)
		\le \limsup_{j\to\infty} \frac{1}{j}\log \Lambda_{2k,2}(\lambda\cap [j,j+1))
 \,,
    \end{equation*}
    and Lemma~\ref{theochar} proves the first claim.\\

Let us suppose now that $\limsup \tfrac{\log\log n}{\lambda_n}=0$ and see that the reverse inequality also holds. Note in  first place that by (the upper bound in) Lemma~\ref{lemenI} we have
\begin{equation}\label{eqzero}
\begin{split}
&
		\limsup_{j\to\infty} \frac{\log \Lambda_{2k,2}(\lambda\cap [j,j+1))}{j}
		\\ &\le \limsup_{j\to\infty} \Big(\frac{1}{2j}\log \big( C\log \big(\#(\lambda\cap [j,j+1))+2\big)\big) + \frac{1}{2j} \log \sup_{A\subseteq \lambda\cap [j,j+1)} \frac{E_k(A)^{1/k}}{\#A}\Big) 
		\\ &\le \limsup_{j\to\infty} \frac{1}{2j}\log\log \big(\#(\lambda\cap [j,j+1))+2\big)
		+\limsup_{j\to\infty} \sup_{A\subseteq \lambda\cap [j,j+1)} \frac{1}{2j} \log  \frac{E_k(A)^{1/k}}{\#A}.
\end{split}
\end{equation}
Now notice that for every $\delta>0$ there exists an $n_0$ such that
\[
\log \log n \le \delta \lambda_n \,,
\]
for every  $n\ge n_0$. For each integer $j> \lambda_{n_0}$, if $j\le \lambda_n< (j+1)$, then
\[
n\le e^{e^{\delta \lambda_n}}\le e^{e^{\delta (j+1)}} \,.
\]
In particular,
\[
\#(\lambda\cap [j,j+1))\le e^{e^{\delta (j+1)}} \,.
\]
So we get
\[
	\limsup_{j\to\infty} \frac{1}{2j}\log\log \big(\#(\lambda\cap [j,j+1))+2\big)
\le \limsup_{j\to\infty} \frac{j+1}{2j}\delta= \frac{\delta}{2}.
\]
Since $\delta$ was arbitrary, the result follows joining \eqref{eqzero}  and again Lemma~\ref{theochar}.
\end{proof}

\begin{proof}[Proof of Theorem~\ref{corchar}]
First notice that \ref{itii} $\Leftrightarrow$ \ref{itiii} follows from Lemmas \ref{theochar} and \ref{lemen}. It is also clear that \ref{iti} implies \ref{itii}. Regarding the converse implication, by Corollary~\ref{corlil} we see that $\tau_{\sigma}$ 
maps $\mathcal{H}_{1}(\lambda)$ into $\mathcal{H}_{q}(\lambda)$ for every $1\le q<\infty$ and every $\sigma>0$. Since $\hh_p(\lambda)\subseteq \hh_1(\lambda)$ for
every $1\le p<\infty$, hypercontractivity follows.
\end{proof}

\subsection{Strips}
Next we turn our attention to Theorem~\ref{rem} and Proposition~\ref{propext}.

\begin{proof}[Proof of Theorem~\ref{rem}]
First assume that $L(\lambda)<\infty$. The lower bound in \eqref{rem1} as well as the upper bound in \eqref{rem2} are known, since  $S_p(\lambda)$ is decreasing in $p$ and (recall \eqref{izagirre}) $S_2(\lambda) = L(\lambda)/2$.
The upper bound in \eqref{rem1} can be deduced from \eqref{eqlip}, but we give a simpler more direct argument. For $\sigma> L(\lambda)/p$ and $D\in \mathcal{H}_p(\lambda)$  with $1\le p\le 2$ we get,
\[
\sum_{n} |a_n|e^{-\lambda_n \sigma}\le
\Big(\sum_{n} |a_n|^{p'}\Big)^{1/p'} \Big(\sum_{n} e^{-\lambda_n p\sigma} \Big)^{1/p}
\le \|D\|_{p} \Big(\sum_{n} e^{-\lambda_n p\sigma} \Big)^{1/p}<\infty,
\]
where in the  penultimate step  we use the Hausdorff-Young inequality \eqref{hy1}.
Let us check that $S_p(\lambda)\ge S_2(\lambda) = L(\lambda)/2$ for $p\ge 2$. Take a Dirichlet series $D=\sum a_n e^{-\lambda_n s}\in \hh_{2}(\lambda)$ and $\sigma>S_p(\lambda)$ (recall that $S_p(\lambda)\le L(\lambda)/2<\infty$). With this at hand, for $N\in\N$ and  a sequence of independent Bernoulli random variables $\varepsilon_n=\pm 1$, define the random Dirichlet series
\[
D_{N,\varepsilon}= \sum_{n=1}^N a_n \varepsilon_n e^{-\lambda_n s}\,.
\]
Then by~\eqref{closedgraph} for  each choice of signs,
\[
\sum_{n=1}^N |a_n|e^{-\lambda_n \sigma}\le C_\sigma \|D_{N,\varepsilon}\|_p\,.
\]
Averaging and applying the Kahane-Khinchin inequality we get
\[
\sum_{n=1}^N |a_n|e^{-\lambda_n \sigma}\le C_\sigma\E_\varepsilon \big[\|D_{N,\varepsilon}\|_p^p\big]^{1/p}\le C_\sigma C_p \Big(\sum_{n=1}^N|a_n|^2\Big)^{1/2}\le C_\sigma C_p \|D\|_2\,.
\]
Letting $N\to \infty$ and since $D\in  \hh_{2}(\lambda)$ was arbitrary, we see that $\sigma> S_2(\lambda) = L(\lambda)/2$.

Regarding \eqref{eqlip}, notice that from the definition of the strips and Corollary~\ref{corlil}, for $p<q$ we have
\begin{equation}   \label{eqlip2}
    S_p(\lambda)\le S_q(\lambda)+T_{q,p}(\lambda)\le S_q(\lambda)+p(\tfrac{1}{p}-\tfrac{1}{q}) T_{\infty,p}(\lambda)\le S_q(\lambda)+(1-\tfrac{p}{q}) S_p(\lambda) \,.
\end{equation}
Rearranging we get
\[
pS_p(\lambda)\le qS_q(\lambda)\,.
\]
Moreover, from \eqref{eqlip2} and the fact that $S_p(\lambda)$ is decreasing we have
\[
0\le  S_p(\lambda) - S_q(\lambda) \le \frac{S_p(\lambda)}{q}(q-p)\,.
\]
So, $S_p(\lambda)$ is Lipschitz as a function of $p$ and therefore it is differentiable almost everywhere. Dividing by $q-p$ and letting $q \to p$ we get \eqref{eqlip}.
This completes the proof under the assumption that $L(\lambda)<\infty$.

It remains to show that $L(\lambda)<\infty$ whenever $S_p(\lambda)<\infty$ for some $p$.
Since $S_p(\lambda)$ is decreasing in $p$, without loss of generality we assume that $S_p(\lambda)<\infty$ for some $p\ge 2$ and let $\sigma>S_p(\lambda)$. For $j\in \N$, we have
\begin{multline*}
	\#(\lambda\cap [j,j+1)) \le e^{(j+1)\sigma} \sum_{n: \lambda_{n}\in [j,j+1)} e^{-\lambda_n \sigma}
\\
	\le C_\sigma e^{(j+1)\sigma} \Big\|\sum_{n: \lambda_{n}\in [j,j+1)} e^{-\lambda_n s} \Big\|_p
	\le C_\sigma e^{(j+1)\sigma} \#(\lambda\cap [j,j+1))^{1/p'},
\end{multline*}
where in the last step we used  the  Hausdorff-Young inequality \eqref{hy2}. Rearranging the last expression we get
\[
\#(\lambda\cap [j,j+1))\le C_{\sigma,p} e^{jp\sigma}\,.
\]
So, by Lemma~\ref{lella},
\[
L(\lambda)=\limsup_{j\to\infty} \frac{\log \#(\lambda\cap [j,j+1))}{j}\le p\sigma <\infty \,.\qedhere 
\]
\end{proof}

\begin{proof}[Proof of Proposition~\ref{propext}] We write $p_0=p_0(\lambda)$ to alleviate notation. Without loss of generality, assume that $q>p$. Fix $p\le r \le p_0$. By Corollary~\ref{corlil} and Theorem~\ref{rem}, we have
\[
\frac{T_{q,p}(\lambda)}{ \tfrac{1}{p}-\tfrac{1}{q}} \le pT_{\infty,p}(\lambda)
\le p S_{p}(\lambda)\le r S_{r}(\lambda)
\le p_0 S_{p_0}(\lambda)=\frac{p_0 L(\lambda)}{2}\,.
\]
A careful look at the previous chain of inequalities shows that \eqref{eqext}, \ref{it2} $\Rightarrow$ \ref{it1} and \ref{it1} $\Rightarrow$ \ref{it3} hold.

It only remains to check that \ref{it3} $\Rightarrow$ \ref{it4}, since clearly \ref{it4} $\Rightarrow$ \ref{it2}. 
Choose $s$ such that $p \leq r < s \leq  p_0$ and $s \le q$.
Then by Remark~\ref{bilbao} and Corollary~\ref{corlil},
\begin{equation*}
        \frac{p_0 L(\lambda)}{2}(\tfrac{1}{r}-\tfrac{1}{s})= S_r(\lambda) - S_{s}(\lambda)\le T_{s,r}(\lambda) \le \frac{\tfrac{1}{r}-\tfrac{1}{s}}{\tfrac{1}{r}-\tfrac{1}{q}}T_{q,r}(\lambda)  \,,
\end{equation*}
which together with \eqref{eqext} gives \ref{it4}.
    
To conclude the proof, note that if $S_1(\lambda)=L(\lambda)$, then \eqref{eqext} implies that $p_0 =2$. Then \ref{it1} holds for $p=1$, and all the strips involved are determined.
\end{proof}

\subsection{Optimality}
\label{opti}
Now we turn to Propositions \ref{propshift} and \ref{example}, that deal with the optimality of the bounds in Proposition~\ref{cormax} and Theorem~\ref{rem}. The proofs are inspired by  Bayart's example from \cite[Theorem~5.2]{Ba}, a very structured frequency composed of packets of arithmetic progressions that attains the upper bounds in our problems. The range of all possible values can then be achieved by combining this frequency with an unstructured $\Q$-li frequency and balancing the proportion of independent vs. structured points.

We start with the following preliminary result.
\begin{lemma}\label{lemsup}
For every pair of frequencies $\mu,\nu$, we have
\[
L(\mu\cup\nu)=\max\{L(\mu),L(\nu)\}\,.
\]
If, addictionally $\langle\mu\rangle_\Q\cap \langle\nu\rangle_\Q=0$, then the following statements hold:
\begin{enumerate}[(a)]
\item $S_p(\mu\cup\nu)=\max\{S_p(\mu),S_p(\nu)\}$  for every $1\le p \le 2$;
\item for every $1\le p \le q \le \infty$ and $\sigma>0$, the operator $\tau_\sigma$ maps $\hh_p(\mu \cup \nu)$ into $\hh_q(\mu \cup \nu)$ if and only if it separately does so for $\mu$ and $\nu$.
\end{enumerate}
\end{lemma}
\begin{proof}
	From Lemma~\ref{lella} we see that
	\begin{align*}
		\max\{L(\mu),L(\nu)\}\le \limsup_{j\to\infty} \frac{\log \#\big((\mu\cup\nu)\cap [j,j+1)\big)}{j}=L(\mu\cup\nu).
	\end{align*}
Conversely,
\begin{align*}
 L(\mu\cup\nu)&=\limsup_{j\to\infty} \frac{\log \#\big((\mu\cup\nu)\cap [j,j+1)\big)}{j}
 \\&\le \limsup_{j\to\infty} \frac{\log \big(2\max\{\#(\mu\cap [j,j+1)),\#(\nu\cap [j,j+1))\}\big)}{j}
 \\&\le \max\Big\{\limsup_{j\to\infty} \frac{\log \#(\mu\cap [j,j+1))}{j}, \limsup_{j\to\infty} \frac{\log \#(\nu\cap [j,j+1))}{j}\Big\}
 \\&=\max\{L(\mu),L(\nu)\} .
\end{align*}
It is clear that
\[
\max\{S_p(\mu),S_p(\nu)\}\le S_p(\mu\cup\nu)\,.
\]
and that if  $\tau_\sigma$ maps $\hh_p(\mu \cup \nu)$ into $\hh_q(\mu \cup \nu)$ then it separately does so for $\mu$ and $\nu$. So, it remains to prove the converse statements.

Assume without loss of generality that $\mu_1,\nu_1\neq 0$.
It follows from \cite[Theorem~3.28]{DeSch} and the fact that  $\langle\mu\rangle_\Q\cap \langle\nu\rangle_\Q=0$ that there are compact abelian groups $G,H$ with Haar measures $dg,dh$ and characters $\xi_{\mu_n}\in\widehat{G}$, $\zeta_{\nu_n}\in \widehat{H}$ such that for every  $D=\sum a_n e^{-\mu_n s} +\sum b_n e^{-\nu_n s}\in \hh_p(\mu \cup \nu)$,
\[
\|D\|_{\hh_p(\mu\cup \nu)}^p=\int_G \int_H\Big|\sum a_n \xi_{\mu_n}(g) + \sum b_n \zeta_{\nu_n}(h) \Big|^pdg dh\,.
\]
We sketch the argument using the language of \cite{DeSch}. 
Denote by $\mathbb{Q}_d$ the group of all rationals endowed with the discrete topology, and by $\mathbb{Q}_d^\infty$ its countable product. Moreover, write $\mu = (R_1,b_1)$, where $R_1$ is a Bohr matrix and $b_1$ a rationally independent basis for
$\mu$. Similarly, we decompose  $\nu= (R_2,b_2)$. Consider now the mapping 
\[
\gamma: \mathbb{Q}_d^\infty \times \mathbb{Q}_d^\infty \to \mathbb{R} \, \text{ given by } \,
\gamma(\alpha_1, \alpha_2)= \langle \alpha_1,b_1 \rangle + \langle \alpha_2,b_2 \rangle \,.
\]
 The assumption on $\mu$ and $\nu$
ensures  that this is an injective homomorphism, and then the dual of it, $\beta: \mathbb{R} \to  \widehat{\mathbb{Q}_d^\infty} \times \widehat{\mathbb{Q}_d^\infty}$, is a   homomorphism with dense range. 
Finally, it follows that the pair $( \widehat{\mathbb{Q}_d^\infty} \times \widehat{\mathbb{Q}_d^\infty}, \beta)$ serves as  a $(\mu \cup \nu)$-Dirichlet group for $\hh_p(\mu \cup \nu)$.

The characters  $\xi_{\mu_n}\in\widehat{G}$ and $\zeta_{\nu_n}\in \widehat{H}$ are non-trivial by the assumption $\mu_1,\nu_1\neq 0$, and so,
\begin{multline*} 
\Big\|\sum a_n e^{-\mu_n s}\Big\|_{\hh_p(\mu)} =\Big( \int_G \Big|\sum a_n \xi_{\mu_n}(g) \Big|^pdg \Big)^{1/p}\\ 
=\Big(\int_G \Big|\sum a_n \xi_{\mu_n}(g) + \int_H\sum b_n \zeta_{\nu_n}(h)dh \Big|^pdg \Big)^{1/p}
\le\|D\|_{\hh_p(\mu\cup \nu)} \,.
\end{multline*}
Hence,
\[
\max\Big\{\Big\|\sum a_n e^{-\mu_n s}\Big\|_{\hh_p(\mu)}, \Big\|\sum b_n e^{-\nu_n s}\Big\|_{\hh_p(\nu)}\Big\} \le\|D\|_{\hh_p(\mu\cup \nu)} \,.
\]
So if $\sigma>\max\{S_p(\mu),S_p(\nu)\}$,
\[
\sum |a_n| e^{-\mu_n \sigma}+\sum |b_n| e^{-\nu_n \sigma}\le C_\sigma \Big\|\sum a_n e^{-\mu_n s}\Big\|_{p}+ C'_\sigma \Big\|\sum b_n e^{-\nu_n s}\Big\|_{p}\le C''_\sigma\|D\|_{p}\,.
\]
Since $\sigma>\max\{S_p(\mu),S_p(\nu)\}$ was arbitrary, we get
\[
S_p(\mu\cup\nu)\le\max\{S_p(\mu),S_p(\nu)\} \,. 
\]
Finally, assume that $\tau_\sigma:\hh_p(\mu)\to\hh_q(\mu)$ is bounded and the same holds for $\nu$. For every  $D=\sum a_n e^{-\mu_n s} +\sum b_n e^{-\nu_n s}\in \hh_p(\mu \cup \nu)$,
\begin{multline*}
	\|\tau_\sigma D\|_q \le \Big\|\sum a_n e^{-\mu_n \sigma} e^{-\mu_n s}\Big\|_{\hh_q(\mu)}+ \Big\|\sum b_n e^{-\nu_n \sigma} e^{-\nu_n s}\Big\|_{\hh_q(\nu)}\\ 
	\le C  \Big(\Big\|\sum a_n  e^{-\mu_n s}\Big\|_{\hh_p(\mu)}+ \Big\|\sum b_n  e^{-\nu_n s}\Big\|_{\hh_p(\nu)}\Big)
\le 2C 	\| D\|_p. \qedhere
\end{multline*}
\end{proof}

From now on we let $\mu$ be a frequency satisfying the following properties:
\begin{itemize}
	\item $\mu=\{\mu_{j,k}\}_{j\in\N,0\le k \le  2^{j} };$
	\item $\{\mu_{j,k}\}_k \subseteq  [j,j+1)$ for every $j\in\N$;
	\item $\mu_{j,k}=\mu_{j,0}+k \delta_j$ for some $\delta_j>0$;
	\item $\{\mu_{j,0}\}_j\cup \{\delta_j\}_j\cup \{2\pi\}$ is $\Q$-li.
\end{itemize}
This is Bayart's example from \cite[Sections 4 and 5]{Ba}.

As a starting point for the proof of Proposition~\ref{propshift}, we describe the behaviour of $\tau_\sigma$ for $\mu$.

\begin{lemma}\label{lemext} The following statements hold:
\begin{enumerate}[(a)]
\item \label{lemext1} If $1 \leq p \le q \leq \infty$ and $\tau_\sigma$ maps $\hh_p(\mu)$ into $\hh_q(\mu)$, then $\sigma\ge (\tfrac{1}{p}-\tfrac{1}{q})L(\mu)$;

\item \label{lemext2} If $1  \le p \leq 2$ and $p \leq q < \infty $, then for $\sigma \ge (\tfrac{1}{p}-\tfrac{1}{q})L(\mu)$, the operator
$\tau_\sigma$ maps $\hh_p(\mu)$ into $\hh_q(\mu)$. 
\end{enumerate}
\end{lemma}
\begin{proof}
Notice that $L(\mu)=\log 2$ by Lemma~\ref{lella}, since $\#(\lambda\cap [j,j+1))=2^j+1$ for every $j\in\N$. 
 
We start with \ref{lemext2}. By Proposition \ref{cormax} we can assume that $\sigma=(\tfrac{1}{p}-\tfrac{1}{q})\log 2$ without loss of generality. From \cite[Theorem~3.31]{DeSch} and the structure of $\mu$, for a Dirichlet polynomial
\[
D=\sum_{j}\sum_{k=0}^{2^{j}} a_{j,k} e^{-\mu_{j,k} s} \,,
\]
and every $1\le r<\infty$ we have
\begin{align}
    \label{eqboh}
    \|D\|_r^r=\int_{\T^\infty}\int_{\T^\infty} \Big|\sum_{j}\sum_{k=0}^{2^{j}} a_{j,k} w_j z_j^k\Big|^r dw dz \,.
\end{align}
Indeed, using again the language of \cite{DeSch} we have that  $\mu = (R,b)$, where the  Bohr matrix $R$  is a `diagonal block matrix'  of  blocks  $(r_{kl})_{0 \leq k \le 2^j,l=1,2}$ with 
$r_{k1}= 1 $ and $r_{k2}=  k$, that is 
\[ 
R= 
\begin{bmatrix}
\boldsymbol{1} & \boldsymbol{0} & 0& 0& 0& 0& \cdots\\
\boldsymbol{1} & \boldsymbol{1} & 0& 0& 0& 0& \cdots\\
0 & 0 & \boldsymbol{1} & \boldsymbol{0}& 0&  0&\cdots\\
0 & 0 & \boldsymbol{1} & \boldsymbol{1}& 0&  0&\cdots\\
0 & 0 & \boldsymbol{1} & \boldsymbol{2}& 0&  0& \cdots\\
0 & 0 & 0 & 0& \boldsymbol{1} & \boldsymbol{0}&  \cdots\\
0 & 0 & 0 & 0& \boldsymbol{1} & \boldsymbol{1}&  \cdots\\
0 & 0 & 0 & 0& \boldsymbol{1} & \boldsymbol{2}&  \cdots\\
0 & 0 & 0 & 0& \boldsymbol{1} & \boldsymbol{3}&  \cdots\\
0 & 0 & 0 & 0& \boldsymbol{1} & \boldsymbol{4}&  \cdots\\ 
\vdots & \vdots & \vdots & \vdots& \vdots & \vdots& \ddots
\end{bmatrix}
\]
and  $b = (\mu_{1,0}, \delta_1,\mu_{2,0}, \delta_2, \mu_{3,0}, \delta_3, ...)$ the $\Q$-li. basis. By \cite[Theorem 3.31]{DeSch} we get \eqref{eqboh}.

 Applying Khinchin's and Minkowski's integral inequalities
  (or, directly  the well known facts that $H_r(\T^\infty)$ has cotype $2$ for $1 \leq r \leq 2$ and type $2$ for $2 \leq r < \infty$), for some constants $C_r>0$ depending on $r$ we have 
\begin{align}\label{khin}
		\|D\|_r \le C_r \Big(\int_{\T^\infty} \Big(\sum_{j}\Big|\sum_{k=0}^{2^{j}} a_{j,k} z_j^k\Big|^2\Big)^{r/2} dz\Big)^{1/r} \le  C_r \Big(\sum_j \Big\|\sum_{k=0}^{ 2^{j}} a_{j,k} z^k\Big\|_{r}^2\Big)^{1/2},
  \intertext{if $2\le r <\infty$, and,}
  \label{khin2}
\Big(\sum_j \Big\|\sum_{k=0}^{ 2^{j}} a_{j,k} z^k\Big\|_{r}^2\Big)^{1/2}
\le  \Big(\int_{\T^\infty} \Big(\sum_{j}\Big|\sum_{k=0}^{2^{j}} a_{j,k} z_j^k\Big|^2\Big)^{r/2} dz\Big)^{1/r} \leq C_r \|D\|_r \,,
	\end{align}
	if $1\le r \le 2$. This allows us to use Nikolskii's inequality 
 from  Lemma~\ref{N-type} which, together with Lemma~\ref{monotone2}, give for  $j\in\N$
	\begin{multline*}
		 \Big\|\sum_{k=0}^{ 2^{j}}  a_{j,k} e^{-\mu_{j,k} \sigma} e^{-\mu_{j,k} s}\Big\|_q 
		 \le  2^{(j+1)(1/p-1/q)}\Big\|\sum_{k=0}^{ 2^{j}}  a_{j,k} e^{-\mu_{j,k} \sigma} e^{-\mu_{j,k} s}\Big\|_p
		 \\  \le 2^{(j+1)(1/p-1/q)} e^{-j \sigma} \Big\|\sum_{k=0}^{ 2^{j}}  a_{j,k}  e^{-\mu_{j,k} s}\Big\|_p
		  \le  2    \Big\|\sum_{k=0}^{ 2^{j}}  a_{j,k}  e^{-\mu_{j,k} s}\Big\|_p.
	\end{multline*}
Combining this with \eqref{khin}, \eqref{khin2} and the density of Dirichlet polynomials in $\hh_r(\lambda)$ for $1\le r<\infty$ shows that $\tau_\sigma$ is bounded for $q\ge 2$. 

If $p<q<2$ we can reduce the problem to the case $q=2$ as follows. 
Using the well known fact 
that 
$\|g(r \boldsymbol{\cdot})\|_q \leq \|g\|_q$ for $\,0 < r \leq 1 $ and $g\in H_q(\T) = H_q(\mathbb{D})$,
for each coordinate $z_j$ and Jensen's inequality, we have
\begin{equation} \label{eqpoi}
\begin{split}
    \|\tau_\sigma D\|_q &= \Big( \int_{\T^\infty}\int_{\T^\infty} \Big|\sum_{j} \sum_{k=0}^{2^{j}} a_{j,k} e^{- (\mu_{j,0}+k\delta_j) \sigma} w_j z_j^k\Big|^q dw dz\Big)^{1/q}
    \\ &\le \Big( \int_{\T^\infty}\int_{\T^\infty} \Big|\sum_{j} \sum_{k=0}^{2^{j}} a_{j,k} e^{- \mu_{j,0}\sigma} w_j z_j^k\Big|^q dw dz\Big)^{1/q} 
    \\ & \leq\Big(\int_{\T^\infty} \Big(\sum_{j}e^{- 2j\sigma}\Big|\sum_{k=0}^{2^{j}} a_{j,k} z_j^k\Big|^2\Big)^{q/2} dz\Big)^{1/q} \,.
\end{split}
\end{equation}
For every $j\in\N$ define
\[
f_j(z)=\sum_{k=0}^{2^{j}} a_{j,k} z_j^k\,.
\]
Consider the parameters
\[
\alpha=\frac{q\beta}{2}, \quad \beta=\frac{2-p}{q-p}\,,
\]
and observe that they are both greater than 1.
Applying H\"older's inequality twice to \eqref{eqpoi} and using \eqref{khin2} we get 
\begin{align*}
    \|\tau_\sigma D\|_q 
    &\leq   \Big(\int_{\T^\infty} \Big(\sum_{j}e^{- 2j\sigma \alpha}|f_j|^2\Big)^{q/2\alpha}  \Big(\sum_{j}|f_j|^2\Big)^{q/2\alpha'}dz\Big)^{1/q}
    \\&\leq   \Big(\int_{\T^\infty} \Big(\sum_{j}e^{- 2j\sigma \alpha}|f_j|^2\Big)^{q\beta/2\alpha}  dz\Big)^{1/q \beta}
    \Big(\int_{\T^\infty} \Big(\sum_{j}|f_j|^2\Big)^{q\beta'/2\alpha'}dz\Big)^{1/q\beta'} \notag
    \\&=  \Big(\int_{\T^\infty} \sum_{j}e^{- 2j(\tfrac{1}{p}-\tfrac{1}{2})\log 2 }|f_j|^2  dz\Big)^{1/q \beta}
    \Big(\int_{\T^\infty} \Big(\sum_{j}|f_j|^2\Big)^{p/2}dz\Big)^{1/q\beta'} \notag
    \\&\le C  \Big(\int_{\T^\infty} \sum_{j}e^{- 2j(\tfrac{1}{p}-\tfrac{1}{2})\log 2 }|f_j|^2  dz\Big)^{1/q \beta}
    \|D\|_p^{p/q\beta'} \,, 
\end{align*}
where the constant $C$ depends on $p$ and $q$. The integral on the right-hand side corresponds to the case $q=2$ we already dealt with. So we arrive at (possibly changing $C$)
\[
    \|\tau_\sigma D\|_q   \leq C   \|D\|_p^{2/q \beta} \|D\|_p^{p/q\beta'} = C \|D\|_p \,.
\]

We prove now \ref{lemext1} and assume that  $\tau_\sigma$ is bounded. Notice that again by  Lemma~\ref{monotone2} (or by hand if $q=\infty$) for every $j\in \N$ we have
\[
	\Big\|\sum_{k=0}^{ 2^{j}}  e^{-\mu_{j,k} \sigma} e^{-\mu_{j,k} s}\Big\|_q \ge  e^{-(j+1) \sigma}\Big\|\sum_{k=0}^{ 2^{j}}   e^{-\mu_{j,k} s}\Big\|_q
=e^{-(j+1) \sigma}\Big\|\sum_{k=0}^{ 2^{j}}  z^k\Big\|_q
\ge c e^{-(j+1) \sigma} 2^{j/q'},
\]
and,
\[
\Big\|\sum_{k=0}^{ 2^{j}}  e^{-\mu_{j,k} s}\Big\|_p =\Big\|\sum_{k=0}^{ 2^{j}}  z^k\Big\|_p
\le \begin{cases}
    C  2^{j/p'} & \text{if } p>1
    \\ C j & \text{if } p=1 \,.
\end{cases}
\]
In particular, for every $j\in \N$ we get
\[
c e^{-(j+1) \sigma} 2^{j/q'}\le \|\tau_\sigma\|  C  2^{j/p'} j\,.
\]
So we must have
\[
\log 2/q'-\sigma\le \log 2 / p'\,.
\]
Rearranging the expression shows that  $\sigma\ge (\tfrac{1}{p}-\tfrac{1}{q})L(\mu)$.
\end{proof}
As we already mentioned in the introduction, an analogous proof gives an affirmative answer to \cite[Question 5.12]{Ba}.

\begin{proof}[Proof of Proposition~\ref{propshift}]
	Let $\nu$ be a $\Q$-li frequency such that $ \alpha L(\nu)=  L(\mu)$ and moreover $\langle\mu\rangle_\Q\cap \langle\nu\rangle_\Q=0$, and set $\lambda=\mu\cup\nu$. 

Clearly if $\tau_\sigma$ maps $\hh_p(\mu \cup \nu)$ into $\hh_q(\mu \cup \nu)$ then it does so for $\mu$. By Lemma~\ref{lemsup} we have that $L(\lambda)=L(\nu)$. Since
\begin{equation}\label{eqalp}
  (\tfrac{1}{p}-\tfrac{1}{q})L(\mu)=\alpha (\tfrac{1}{p}-\tfrac{1}{q}) L(\nu)=\alpha (\tfrac{1}{p}-\tfrac{1}{q}) L(\lambda)\,,
\end{equation}
applying Lemma~\ref{lemext} we immediately get \ref{propshifta}.

Regarding \ref{propshiftb}, it follows from \eqref{eqalp} and Lemma~\ref{lemext} that $\tau_\sigma$ maps $\hh_p(\mu)$ into $\hh_q(\mu)$. Again by Lemma~\ref{lemsup}, $\tau_\sigma$ also does so for $\mu \cup \nu$ (recall that  $\nu$ is hypercontractive and we are assuming $q<\infty$).
\end{proof}

\begin{proof}[Proof of Proposition~\ref{example}]
We take the same example as in the proof of Proposition~\ref{propshift} for  $\alpha=p_0/2$. In particular, $L(\lambda)=L(\nu)$ and for every $1\le p\le 2$ and $p\le q <\infty$,
\begin{equation}    \label{eqalso}
    T_{q,p}(\lambda)=\frac{p_0}{2}L(\lambda)\Big(\frac{1}{p}-\frac{1}{q}\Big).
\end{equation}
Also, by Lemma~\ref{lemsup} and Theorem~\ref{rem}, 
\[
 S_p(\lambda)=\max\{S_p(\nu),S_p(\mu)\}
 \le \max\Big\{\frac{L(\nu)}{2},\frac{L(\mu)}{p}\Big\}= \max\Big\{\frac{L(\lambda)}{2},\frac{p_0 L(\lambda)}{2p}\Big\}\,.
\]
In particular, we have that $S_{p_0}(\lambda)=L(\lambda)/2$ and so $p_0\ge p_0(\lambda)$ (see \eqref{defpo} for the definition). From this and \eqref{eqalso}, we get
\[
T_{q,p}(\lambda)\ge\frac{p_0(\lambda)}{2}L(\lambda)\Big(\frac{1}{p}-\frac{1}{q}\Big)\,.
\]
which in view of Proposition~\ref{propext} gives the result.
\end{proof}

%
%\bibliographystyle{abbrv}
%\bibliography{biblio}

\end{document}